\documentclass{amsart}
\usepackage{mathtools, amsmath, amssymb, amscd, amsthm, pdfsync}
\usepackage[margin=1in]{geometry}
\usepackage{xcolor, epsfig, psfrag}

\newtheorem{thm}{Theorem}[section]

\theoremstyle{plain}
\newtheorem{lem}[thm]{Lemma}
\newtheorem{cor}[thm]{Corollary}
\newtheorem{prop}[thm]{Proposition}

\newtheorem{quest}[thm]{Question}

\theoremstyle{definition}
\newtheorem{definition}[thm]{Definition}

\theoremstyle{remark}
\newtheorem{eg}[thm]{Example}
\newtheorem{rem}[thm]{Remark}
\newtheorem{Claim}[thm]{Claim}
\newtheorem{prob}[thm]{Problem}
\makeatletter
\newenvironment{proofof}[1]{\par
  \pushQED{\qed}%
  \normalfont \topsep6\p@\@plus6\p@\relax
  \trivlist
  \item[\hskip\labelsep
        \bfseries
    Proof of #1\@addpunct{.}]\ignorespaces
}{%
  \popQED\endtrivlist\@endpefalse
}
\makeatother

\DeclareSymbolFont{AMSb}{U}{msb}{m}{n}
\DeclareMathSymbol{\Z}{\mathbin}{AMSb}{"5A}
\DeclareMathSymbol{\R}{\mathbin}{AMSb}{"52}
\DeclareMathSymbol{\N}{\mathbin}{AMSb}{"4E}
\DeclareMathSymbol{\Q}{\mathbin}{AMSb}{"51}

\DeclareMathOperator{\conv}{conv}

\def\NP{{\textup{\textsf{NP}}}}
\def\poly{{\textup{\textsf{P}}}}
\def\AP{\text{AP}}
\def\Div{\text{Div}}
\def\Equalp{\text{Equal-p}}
\def\CongM{\text{Cong-M}}
\def\ex{\exists}
\def\for{\forall}

\newcommand{\floor}[1]{\lfloor#1\rfloor}
\newcommand{\abs}[1]{\lvert#1\rvert}

\def\bft{\mathbf{t}}

\def\w{\mathbf{w}}
\def\x{\mathbf{x}}
\def\y{\mathbf{y}}

\def\u{\mathbf{u}}
\def\t{\mathbf{t}}
\def\v{\mathbf{v}}

\renewcommand{\mod}[1]{
\;\, (\textup{mod} \; #1)
}
\def\mmod{\; \textup{mod} \;}

\newcommand{\cj}[1]{\overline{#1}}

\def\.{\hskip.06cm}
\def\ts{\hskip.03cm}

\def\del{\delta}

\def\L{\mathcal{L}}
\def\C{\mathcal{C}}

\def\wt{\widetilde}
\def\const{\textup{const}}

\def\ov{\overline}

\newcommand{\norm}[1]{\| #1 \|}

\begin{document}
%%\title[Complexity of counting points in parametric Presburger]{On the complexity of counting points in families definable in parametric expansions of Presburger arithmetic}
\title{Parametric Presburger arithmetic: Complexity of counting and quantifier elimination}
\author{Tristram Bogart, John Goodrick, Danny Nguyen, and Kevin Woods}

\begin{abstract}
We consider an expansion of Presburger arithmetic which allows multiplication by $k$ parameters $t_1,\ldots,t_k$. A formula in this language defines a parametric set $S_\t \subseteq \Z^{d}$ as $\t$ varies in $\Z^k$, and we examine the counting function $\abs{S_\t}$ as a function of $\t$. For a single parameter, it is known that $\abs{S_t}$ can be expressed as an eventual quasi-polynomial (there is a period $m$ such that, for sufficiently large $t$, the function is polynomial on each of the residue classes mod $m$). We show that such a nice expression is impossible with 2 or more parameters. Indeed (assuming $\poly \neq \NP$) we construct a parametric set $S_{t_1,t_2}$ such that  $\abs{S_{t_1, t_2}}$ is not even polynomial-time computable on input $(t_1,t_2)$. In contrast, for parametric sets $S_\t \subseteq \Z^d$ with arbitrarily many parameters, defined in a similar language without the ordering relation, we show that $|S_\t|$ is always polynomial-time computable in the size of $\t$, and in fact can be represented using the gcd and similar functions.
\end{abstract}

\maketitle{}

\section{Introduction}

We study the difficulty of counting points in parametric sets of the form
\begin{equation}\label{eq:param_Presburger}
S_{\t} = \{ \x \in \Z^d \;\; : \; \;Q_{1}y_{1} \; Q_{2}y_{2} \; \dots Q_{m}y_{m} \;\; \Theta_{\t}(\x, \y) \}.
\end{equation}
Here $\t = (t_{1},\dots,t_{k})$ are the \emph{parameters}, $\x = (x_{1},\dots,x_{d})$ are the \emph{free} variables, and $\y = (y_{1},\dots,y_{m})$ are the \emph{quantified} variables, all ranging over $\Z$; $Q_i \in \{\forall, \exists\}$ are the quantifiers; and $\Theta_{\t}(\x, \y)$ is a Boolean combination,  in disjunctive normal form,  of linear inequalities in $\x,\y$ with coefficients in $\Z[\t]$. That is,
%More concretely, by elementary arithmetic and logical operations, the expression $\Theta(\x,\y)$ can be assumed to be in the disjunctive normal form
\begin{equation}\label{eq:DNF}
 \Theta_{\t}(\x, \y) = \left[A_1(\t) \cdot (\x, \y)^T \; \leq \; \cj b_1(\t) \right] \vee \ldots \vee \left[A_\ell(\t) \cdot (\x, \y)^T \leq \cj b_\ell(\t)\right],
\end{equation}
where each $A_i(\t)$ is a $r_{i} \times (d+m)$ matrix, each $\cj b_{i}(\t)$ is a length $r_{i}$ column vector,  all with entries in $\Z[\t]$,  and the concatenation $(\x,\y)$ of the $\x$ and $\y$ variables is treated as a row vector.\footnote{By a simple trick, we do not need to worry about negations $\neg (\lambda_1 x_1 + \ldots + \lambda_{d+m} y_m \leq c)$ of basic inequalities, since these are equivalent to strict inequalities ``$\ldots > c$,'' which in turn are equivalent to non-strict inequalities ``$\ldots \geq c+1$'' since we are working over the integers.}
If there are $k$ parameters $t_1,\ldots,t_k$, we say that the family of sets $\{ S_{\t} : \t \in \Z^k\}$ is a \emph{$k$-parametric Presburger family}. A general expression of the type \begin{equation}\label{eq:kparamPA}
\Phi_\t (\x) = Q_{1}y_{1} \; Q_{2}y_{2} \; \dots Q_{m}y_{m} \;\; \Theta_{\t}(\x, \y)
\end{equation}
with $\Theta_\t(\x, \y)$ as in (\ref{eq:param_Presburger}) is called a \emph{formula in $k$-parametric Presburger Arithmetic} (often abbreviated as $k$-parametric PA). Classic Presburger arithmetic corresponds to $k=0$. 

\begin{quest} \label{quest:big}
Given a $k$-parametric Presburger family defined by $S_\t = \{ \x \in \Z^d : \Phi_t(\x) \}$, under what conditions on the formula $\Phi_{\t}$ is the counting function $\abs{S_{\t}}$ a ``nice'' function of $\t$?\end{quest}

Of course, ``nice'' is a vague qualifier, so let's start with some nice examples. We will assume that the parameters $t_i$ are nonnegative in the following examples, which simplifies the number of cases:

\begin{eg}\label{ex:a}
If we define $S_{t_1,t_2} = \{x \in \Z: \; x\ge0 \ \wedge \ t_1 x \leq t_2\},$ then
\[ |S_{t_1,t_2}| = \floor{t_2/t_1}+1. \]
\end{eg}

\begin{eg}\label{ex:b}
The set $S_{t_1,t_2}=\big\{(x_1,x_2)\in\Z^2:\ x_1,x_2\ge 0\; \wedge \; t_1 x_1 +t_2 x_2=t_1t_2\big\}$
  consists of the integer points on a line segment with endpoints $(t_2,0)$ and $(0,t_1)$, and so
  \[\abs{S_{t_1,t_2}}=\gcd(t_1,t_2)+1.\]
\end{eg}

\begin{eg}\label{ex:c}
If $S_{t_1,t_2} = \{(x_1,x_2) \in \Z^2: \; x_1,x_2\ge0\ \wedge\ x_1+x_2=t_1\ \wedge\ 2x_1+x_2\le t_2\}$, then the equality forces $x_2=t_1-x_1$ (which is only valid if $x_1\le t_1$) and substituting into the inequality shows that
\begin{align*} |S_{t_1,t_2}| &= \abs{\{x_1\in\Z:\ 0\le x_1\le \min(t_1,t_2-t_1)\}}\\
&= \begin{cases}
t_1+1&\text{if } 2t_1\le t_2,\\
t_2-t_1+1 & \text{if } t_1\le t_2 <2t_1,\\
0&\text{if $t_2<t_1$.}
\end{cases}
\end{align*}
\end{eg}

\begin{eg}\label{ex:d}
If $S_{t} = \{x \in \Z: \; \exists y\in \Z,\ x,y\ge 0 \ \wedge \ 2x+2y+2=t\},$ then
\[ |S_{t}| =\begin{cases}
t/2 &\text{if $t$ even, $t\ge 2$},\\
0&\text{if $t$ odd.}
\end{cases} \]\end{eg}

We're seeing many types of ``nice'' functions in these examples, and the question is now how to generalize. In fact, Example \ref{ex:d} generalizes to any family in 1-parametic Presburger arithmetic \cite{BGW2017}, as described in the next section.

\subsection{$1$-parametric Presburger arithmetic}
In the case of a single parameter $t$, our perspective means studying families $\{S_t \; : \: t \in \Z\}$ of subsets of $\Z^d$ of the form
\begin{equation}\label{eq:1_param_Presburger}
S_t = \{ \x \in \Z^d \;\; : \; \;Q_{1}y_{1} \; Q_{2}y_{2} \; \dots Q_{m}y_{m} \;\; \Theta_t(\x, \y) \},
\end{equation}
where $\Theta_t(\x, \y)$ is exactly as in (\ref{eq:DNF}) except that the entries of the $A_i$'s and the $\cj b_i$'s come from the univariate polynomial ring $\Z[t]$. The study of such \emph{1-parametric PA families} was proposed by Woods in \cite{Woods2014}.
%who called them \emph{parametric Presburger families}.
These families were further analyzed in \cite{BGW2017}, in which the main result is that they exhibit \emph{quasi-polynomial} behavior:

\begin{definition}\label{def:EQP}
  A function $g:\Z \rightarrow\Z$ is a \emph{quasi-polynomial} if there exists a period $m$ and polynomials $f_0,\ldots,f_{m-1}\in\Q[t]$ such that
\[g(t)=f_i(t),\text{ for }t\equiv i\bmod m.\]
A function $g:\Z\rightarrow\Z$ is an \emph{eventual quasi-polynomial}, abbreviated \emph{EQP}, if it agrees with a quasi-polynomial for sufficiently large $|t|$.
\end{definition}

Example \ref{ex:d} is a family where $\abs{S_t}$ is an EQP.

\begin{thm}\label{th:1-param} \cite{BGW2017}
  Let $\{S_t : t \in \Z\}$ be a 1-parametric PA family. There exists an EQP $g:\Z\rightarrow\N$ such that, if $S_t$ has finite cardinality, then $g(t)=\abs{S_t}$. The set of $t$ such that $S_t$ has finite cardinality is eventually periodic.
  
\end{thm}

\begin{rem}
  In~\cite{BGW2017}, the parameter $t$ takes values in $\N$ instead of $\Z$.
  However, one can see that the same proofs and conclusions also hold when $t$ ranges over $\Z$.
\end{rem}

There are several other forms of quasi-polynomial behavior that 1-parametric PA families exhibit (such as possessing EQP Skolem functions; see \cite{BGW2017}). Here we focus on the cardinality, $\abs{S_t}$. We hope the reader agrees that EQPs are relatively ``nice'' functions.

\subsection{$k$-parametric Presburger arithmetic}
Let us restate our main definition:

\begin{definition}\label{def:k-PPA}
  A \emph{$k$-parametric PA family} is a collection
  $\{S_\bft : \bft = (t_1, \dots, t_k) \in \Z^k \}$
  of subsets of $\Z^d$ of the form
\begin{equation}\label{eq:multi_Presburger}
S_\bft = \{ \x \in \Z^d \;\; : \; \;Q_{1}y_{1} \; Q_{2}y_{2} \; \dots Q_{m}y_{m} \;\; \Theta_\bft(\x, \y) \},
\end{equation}
where now
%the $Q_i \in \{\forall, \exists\}$ are quantifiers for variables $y_i$ ranging over $\Z$ and
$\Theta_\bft(\x, \y)$ is a Boolean combination of linear inequalities with coefficients in $\Z[\bft]$.

A \emph{$k$-parametric PA formula} $\Phi_{\t}$ is an expression ``$Q_{1}y_{1} \; Q_{2}y_{2} \; \dots Q_{m}y_{m} \; \Theta_\bft(\x, \y)$'' as above, or any logically equivalent first-order formula in the language $\mathcal{L} = \{+, 0, 1, \leq, \lambda_p(\t) : p \in \Z[\t]\}$ with a function symbols for $+$, unary function symbols $\lambda_p(\t)$ for multiplication by each polynomial $p(\t) \in \Z[\t]$, constant symbols for $0$ and $1$, and a relation symbol for $\leq$.

\end{definition}

\begin{rem}
Abusing the notation, we also denote the parametric family $\{S_{\t} : \t \in \Z^{k}\}$ just by $S_{\t}$ when the dimension $k$ is clear.
\end{rem}

Examples \ref{ex:a}, \ref{ex:b}, and \ref{ex:c} show that $k$-parametric PA families, with $k\ge 2$, \emph{can} have nice counting functions, $\abs{S_{\t}}$. Will they always? We despair of defining ``nice'' precisely, but we can at least provide a necessary condition: for a fixed family $S_{\t}$, if $\abs{S_{\t}}$ is to qualify as a nice function, there must at least be a polynomial-time algorithm that takes as input $\t\in\Z^k$ and outputs $\abs{S_{\t}}$.

\begin{quest} \label{quest:main} Given a $k$-parametric Presburger family defined by $S_\t = \{ \x \in \Z^d : \Phi_\t(\x) \}$, under what conditions on the (fixed) formula $\Phi_{\t}$ is the counting function $|S_{\t}|$ polynomial-time computable, taking as input the values of the parameters $\t$?
\end{quest}

Note that we define polynomial-time computation in the usual computer-science sense: the number of steps of the algorithm must be polynomial in the \emph{input size} of $\t$ (that is, the number of bits to encode $\t$ into binary), which is $k+\sum_i \log_2\abs{t_i}$. For example, the Euclidean algorithm is polynomial-time: it computes $\gcd(t_1,t_2)$ in number of arithmetic operations bounded by a degree 1 polynomial in $2+\log_2 t_1 + \log_2 t_2$.

The functions $|S_{\t}|$ from Examples \ref{ex:a} through \ref{ex:d} are all polynomial-time computable. From Theorem \ref{th:1-param} and the observation that EQPs are polynomial-time computable, we immediately obtain an answer to Question \ref{quest:main} in the case of a single parameter $t$:

\begin{cor}\label{cor:1_param}
  Let $S_t$ be any fixed 1-parametric PA family.
  Then there are polynomial time algorithms to: \textup{i)} check if $|S_t| = \infty$, \; \textup{ii)} compute $|S_{t}|$ if $|S_{t}| < \infty$.
\end{cor}

The main goal of this paper is to construct a fixed 2-parametric PA family $\{S_{t_{1},t_{2}} : (t_{1},t_{2}) \in \Z^2\}$ for which there is no polynomial-time algorithm computing $\abs{S_{\t}}$ (assuming $\poly \neq \NP$). Therefore, while we cannot say with precision what a nice function should be like, we can say that this particular counting function $\abs{S_{\t}}$ is not nice. Furthermore, this implies that certain classes of functions (polynomials, gcds, floor functions, modular reductions,\ldots) are not expressive enough to capture $|S_{\t}|$, even for a very simple-looking $S_{\t}$. This contrasts with the 1-parameter case, where $\abs{S_{t}}$ is always an EQP and hence polynomial-time computable.

Definition \ref{def:k-PPA} is a generalization of \emph{classical Presburger arithmetic} (PA), in which a formula $\Phi$ is given only with explicit integer coefficients and constants ($A_{i}$ and $\cj b_{i}$) without any parameters $\t$.
PA is \emph{decidable}, meaning there is an algorithm to decide the truth of any given well-formed sentence in it.
Moreover, PA has full \emph{quantifier elimination} in an expanded language with predicates for divisibility by each fixed integer.
This important logical fact permits an algorithm to actually count the cardinality of any set definable by a PA formula $\Phi$ with an arbitrary number of quantifiers and inequalities, although with an unpractical triply exponential complexity in the length of $\Phi$ (see~\cite{Oppen}).
The complexity of PA is itself a fundamental topic in the study of decidable logical theories and their complexities (see~\cite{FR,Gradel}).

Returning to $k$-parametric PA, for a fixed formula $\Phi_{\t}$, given any value $\textbf{a} \in \Z^{k}$ for $\t$, we can substitute it into $\Phi_{\t}$ to get a formula $\Phi_\textbf{a}$ in PA.
By the above paragraph, the parametric counting problem for~\eqref{eq:param_Presburger} is always computable.
Moreover, the form of the resulting formula $\Phi_{\textbf{a}}$, especially its number of quantifiers and inequalities, stays the same for different values $\textbf{a}$ of $\t$.
So we can hope that the complexity of computing $|S_{\t}|$ (for a fixed family $S_\t$) is much lower than that of counting solutions to a general PA formula (when the formula is not fixed, but instead given as input to the algorithm). To reiterate, it is critical in our analysis that the formula $\Phi_{\t}$ be fixed throughout, and we look for an efficient algorithm with $\t$ as the only input.

\subsection{Summary of results}

Our main result is that if $\poly \neq \NP$ (technically, we only need the weaker assumption that $\#\mathsf{P} \neq \mathsf{FP}$), then there exists a $2$-parametric PA family $S_{\t}$ such that $\abs{S_\t}$ is not polynomial-time computable; in fact, such a family exists with  limited alternation of quantifiers. First we recall the $\Sigma_n$ and $\Pi_n$ hierarchies of first-order formulas based on the number of quantifier alternations.

\begin{definition}
\label{sigma_n}
A $k$-parametric PA formula $\Phi_{\t}(\x)$ is in $\Sigma_1$ (respectively, $\Pi_1$) if it is logically equivalent to one of the form
\[Q_{1}y_{1} \; Q_{2}y_{2} \; \dots Q_{m}y_{m} \;\; \Theta_\bft(\x, \y) \]
in which every quantifier $Q_i$ is $\exists$ (respectively, every $Q_i$ is $\forall$), and $\Theta_\bft(\x, \y)$ is a Boolean combination of linear inequalities with coefficients in $\Z[\bft]$.

Inductively, a $k$-parametric PA formula $\Phi_{\t}(\x)$ is in $\Sigma_{n+1}$ (respectively, $\Pi_{n+1}$) if it is equivalent to one of the form
\[Q_{1}y_{1} \; Q_{2}y_{2} \; \dots Q_{m}y_{m} \;\; \Phi'_{\t}(\x, \y) \]
in which every $Q_i$ is $\exists$ (respectively, $\forall$) and $\Phi'_{\t}(\x, \y)$ is a formula in $\Pi_n$ (respectively, $\Sigma_n$).
\end{definition}

\begin{thm} \label{thm:2PPA}
  Assume $\poly \neq \NP$. There exists a $2$-parametric $\Sigma_2$  PA family $S_{t_{1},t_{2}}$ for which $|S_{t_{1},t_{2}}|$ is always finite but cannot be expressed as a polynomial time evaluable function in $t_{1}$ and $t_{2}$.

%  The cardinality of $2$-parameter $\Sigma_2$ PA formulas $F_{t_{1},t_{2}}$ cannot be expressed as polynomial time evaluable functions in $t_{1}$ and $t_{2}$.
\end{thm}

Two corollaries are:

\begin{cor}\label{cor:2_param}\label{cor:positive_hard}
There is a $2$-parametric family $S_{t_{1},t_{2}}$ such that the set of $(t_{1},t_{2}) \in \Z^{2}$ for which $|S_{t_{1},t_{2}}|$ is positive cannot be described using polynomial-time relations in $t_{1},t_{2}$.
%Thus, there is no easy analogue for Corollary~\ref{cor:2_param} with two parameters.
\end{cor}

\begin{cor}\label{cor:2_param_Presburger}\label{cor:qe_hard}
Any extension of $2$-parametric PA with only polynomial-time computable predicates cannot have full quantifier elimination.
\end{cor}

\subsection{Structure of the rest of the paper}
We will present what amount to two different proofs of Theorem~\ref{thm:2PPA} in the following two sections. In each case, we leverage the main result of Nguyen and Pak \cite{NP} which yields a $3$-parametric $\Sigma_2$ PA formula, and then show how this can be reduced to a $2$-parametric $\Sigma_2$ PA formula whose points are equally ``hard'' to count (modulo polynomial-time reductions). The first reduction we present, in Section 2, uses a trick due to Glivick\'{y} and Pudl\'ak \cite{GP} to encode multiplication by three different integers using multiplication by only two integers, and this reduction has the advantage of not increasing the number of free variables in the formula. Next, in Section 3 we present a more general counting-reduction technique which is less \emph{ad hoc} and reduces any $k$-parametric PA formula to a $2$-parametric PA formula with the same number of quantifier alternations; the idea here is a little more transparent than in Section 2, but it has the disadvantage of introducing many more new free and quantified variables to the formula, so we consider that it is interesting to present both reductions.

%FILL IN (Kevin): relationship to positive results of Barvinok and Barvinok-Woods for any number of parameters and zero or one quantifier. 

%FILL IN (John): relationship to the Holly-van den Dries 2-sorted theories of modules and to Glivicky's "linear arithmetics."

In Section 4 we consider a variant of Question \ref{quest:main} in which there is no order relation in our language; that is, we can only express linear equations but not linear inequalities. Quantifier-free formulas in this language
%in this language of \emph{unordered}, when written in disjunctive normal form,
define finite unions of \emph{lattice translates}. This setting was studied in detail from a model-theoretic perspective by van den Dries and Holly \cite{vdDH}, and we apply their results to show that, in contrast to Theorem \ref{thm:2PPA}, the counting functions in the unordered setting can be computed in polynomial time, regardless of the number of parameters and of quantifier alternations. Indeed, these functions can be expressed using gcd and related functions.

Finally, in Section 5 we discuss the optimality of Theorem~\ref{thm:2PPA} by explaining what happens when we weaken or modify some of the hypotheses.

\bigskip

\section{Proof of Theorem \ref{thm:2PPA} and its corollaries}\label{sec:main_proof}
%\begin{rem}
In what follows, it will be convenient to allow $k$-parametric PA formulas in which the quantifiers are not necessarily outside the scope of all Boolean operations, but these are always logically equivalent to expressions as in (\ref{eq:multi_Presburger}); for instance, $$\exists y_1 \left[\Theta_t(\x, y_1)\right] \wedge \exists y_1 \left[\Theta'_\t(\x, y_1) \right]$$ is equivalent to $$\exists y_1 \, \exists y_2 \, \left[\Theta_\t(\x, y_1) \wedge \Theta'_\t(\x, y_2) \right].$$
%\end{rem}

In~\cite{NP}, certain subclasses of classical PA formulas, called \emph{short PA formulas}, were investigated.
The PA formulas in each such subclass are allowed to have only a bounded number of variables, quantifiers and inequalities (atomic formulas).
The main problem was to classify the complexity (of counting and decision) for those short PA subclasses.
It was proved that a simple subclass with only $5$ variables, $2$ quantifier alternations and $10$ inequalities is $\NP$-complete to decide, and also $\#\mathsf{P}$-complete to count.
Combined with the positive results in~\cite{barvinok94,BW03}, this settled the last open subcase of classical PA complexity problems.
The main reduction in~\cite{NP} started with the following $\NP$-complete problem:

\begin{prob}
  \textsc{AP-COVER}: Given an interval\footnote{All intervals in the paper are over $\Z$, so $[a,b]$ with $a,b \in \R$ should be understood as $[a,b] \cap \Z$.} $[\mu,\nu] \subset \Z$ and $n$ arithmetic progressions
  \[ \AP_{i} = \AP(g_{i},h_{i},e_{i}) \coloneqq \{g_{i}, g_{i} + e_{i}, \dots, g_{i} + h_{i}e_{i}\},\]
  with $1 \le \mu \le \nu$, $g_{i},h_{i},e_{i} \in \Z$, $h_{i} \ge 1$, decide if there exists some $z \in [\mu,\nu] \backslash \bigcup_{i=1}^{n} \AP_{i}$.
\end{prob}

In other words, the problem asks whether there is some element in the interval $[\mu,\nu]$ not covered by the given arithmetic progressions.
The problem is clearly invariant under a translation of both $[\mu,\nu]$ and the $\AP_{i}$'s, so we can assume $\mu = 1$.
Also without affecting the complexity, we can assume that $g_{1} = \nu, h_{1} = 1, e_{1} = 0$, i.e., $\AP_{1} = \{\nu\}$.
The main argument in~\cite{NP} uses continued fractions to construct an integer $M$ and a rational number $p/q$ such that the \emph{best approximations} of $p/q$, in the terminology of continued fractions, encode $\bigcup_{i=1}^{n} \AP_{i}$ modulo $M$.
%%We refer to~\cite{NP} for the exact details.
The main point is that $p/q$ should satisfy $\floor{p/q} = g_{1} = \nu$, so that $[\mu,\nu] = [1,p/q]$, and the formula
%% Now 
%% Thus, the main formula in~\cite{NP} can be rewritten as:
\begin{equation} \label{eq:NP1}
%% (*)\qquad
\gathered
\Phi_{p,q,M}(z) \; = \;  1 \le z \le p/q \; \land \; \ex \y \;\; 
y_{2} \equiv z \mod{M}
\; \land \; 
\floor{p/q} \le y_{2} < p \; \land \; qy_{2} < py_{1} \; \land \; \\ \for \x \;\;\; 
\lnot
\left\{
\begin{matrix}
py_{1} - qy_{2} \; \ge \; px_{1} - qx_{2} \; \ge \;  0\\
y_{2} \;>\; x_{2} \;>\; 0
\end{matrix}
\right\}
\endgathered
\end{equation}
satisfies the property
\begin{equation}\label{eq:equal_intersection}
\{z \in \Z : \Phi_{p,q,M}(z)   \} = [\mu,\nu] \cap (\bigcup_{i=1}^{n} \AP_{i}).
\end{equation}
Thus, the original $\textsc{AP-COVER}$ instance is \emph{not} satisfied if and only if $|S_{p,q,M}| = |[\mu,\nu]| = \floor{p/q}$.
We emphasize that $p,q,M$ can be computed in polynomial time from $\mu,\nu,g_{i},h_{i},e_{i}$.
The meaning behind this formula can be explained as follows.

In Figure~\ref{f:continued_fraction}, the line $y_{2}/y_{1}=p/q$ divides the positive orthant into two parts.
The integer hull of the points strictly below this line and above the horizontal axis form a polyhedron, whose boundary is the (bold) convex polygonal curve $\C$, starting at $(1,0)$ and ending at $(q,p)$.
Denote by $\C_{i}$ the $i$-th edge of $\C$ above the (dotted) horizontal line $y_{2}=g_{1}=\floor{p/q}$.
Then for every $1 \le i \le n$ we have $\AP_{i} = \{y_{2} \mmod M \, : \, (y_{1},y_{2}) \in \C_{i} \}$, and thus
$$\bigcup_{i=1}^{n} \AP_{i} = \{y_{2} \mmod M \, : \, (y_{1},y_{2}) \in \C,\, y_{2} \ge g_{1}\}.$$

\begin{figure}[hbt]
\begin{center}

\psfrag{O}{\small$O$}
\psfrag{1}{\small$(1,0)$}
\psfrag{C}{\small$\C$}
\psfrag{p,q}{\small$(q,p)$}
\psfrag{l}{\small$\frac{y_{2}}{y_{1}} = \frac{p}{q}$}
\psfrag{g}{\small$y_{2} = g_{1} = \floor{p/q}$}

\psfrag{y1}{\small$y_1$}
\psfrag{y2}{\small$y_2$}

\epsfig{file=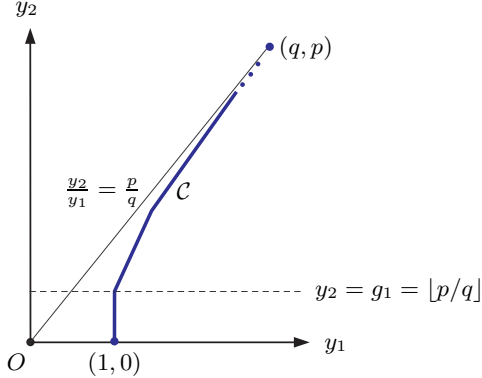,width=5cm}
\end{center}

\vspace{-2em}

\caption{The (bold) sail $\C$ below the line $y_{2}/y_{1}=p/q$.}
\label{f:continued_fraction}
\end{figure}

In \eqref{eq:NP1}, we express $z \in  [\mu,\nu] \cap \big(\bigcup_{i=1}^{n} \AP_{i}\big) $ as $z \equiv y_{2} \mod{M}$ for some $(y_{1},y_{2})$ with $\floor{p/q} \le y_{2} < p$ and  $(y_{1},y_{2}) \in \C$.\footnote{The curve $\mathcal{C}$ includes $(p,q)$ in~\cite{NP}, but not here. This small difference is not very significant as one can easily check.}
By a basic property of continued fractions (see e.g.~\cite{Karpenkov}), the condition $(y_{1},y_{2}) \in \mathcal{C}$ is equivalent to saying that $qy_{2} < py_{1}$, and there is no other integer point $(x_{1},x_{2})$ with $y_{2} > x_{2} > 0$ such that $x_{2}/x_{1}$ approximates $p/q$ better than $y_{2}/y_{1}$.
This last condition is expressed by the $\for \x \dots$  clause in $\Phi_{p,q,M}$.

\medskip

A hardness result for 3-parameter PA immediately follows.

\begin{prop} \label{prop:3params}
Assume $\poly \neq \NP$. There exists a $3$-parametric $\Sigma_2$ PA family $S_{p,q,M}$ such that $|S_{p,q,M}|$ is always finite but cannot be expressed as a polynomial-time evaluable function in $p$, $q$, and $M$. 
\end{prop}

\begin{proof}
  We can clear the integer denominators in \eqref{eq:NP1} by cross multiplications. The condition
  \[ y_{2} \equiv z \mod{M} \]
  can be expressed with existential quantifiers. %(we will prove an analogous statement, Lemma \ref{lem:cong}, in the 2-parameter case)
Thus we obtain a 3-parametric $\Sigma_2$ PA formula $\Phi_{p,q,M}$, which defines a family $S_{p,q,M}$.
The set of satisfying values $z$ is finite by $1 \leq z \leq p/q$.
  Now assume $|S_{p,q,M}|$ is a polynomial-time evaluable function $f(p,q,M)$.
  Then given any AP-COVER instance, we can compute $p,q,M$ in polynomial time from the $\AP_{i}$'s, and then evaluate $f(p,q,M)$ in polynomial time to check whether $f(p,q,M)=\floor{p/q}$.
  This contradicts $\poly \neq \NP$.
\end{proof}

It remains to reduce the three parameters $p, q, M$ to two. To do this, we will adapt a trick of Glivick\'{y} and Pudl\'ak \cite{GP}. Their context is slightly different from ours in that they use nonstandard integers rather than parameters that range over $\Z$, and that their results involve computability rather than complexity. However their key idea and its proof apply in our context.   
The two parameters that will be involved are 
%%\[ \mathbf{t_1 = pM}, \;\;\; \mathbf{t_2 = pqM^2 + M}.\] 
\begin{equation}\label{eq:new_params}
t_1 = pM, \;\;\; t_2 = pqM^2 + M.
\end{equation}

For convenience, we will assume for the rest of Section $2$ that all the parameters in our formulas ($t_1, t_2, p, q,$ and $M$) only take nonnegative integer values. Although in other parts of this paper the parameters are assumed to range over $\Z$, this restriction does not affect the hardness results we are proving here.

\begin{prop} \cite[\S 3.2]{GP} \label{prop:GP} For $0 \leq j < p$, the three multiplications $j \mapsto pM j,\, j \mapsto qM j,\, j \mapsto Mj$ can be defined by using just two multiplications $j \mapsto t_1j$ and $j \mapsto t_2 j$.
\end{prop}  

\begin{proof}
By definition, we have $t_1j = pMj$ for all $j$, so it remains to define the multiplications by $qMj$ and $Mj$ for $0 \leq j < p$. By the division algorithm, for every $j \geq 0$ we can uniquely write
  $$ (pqM^{2}+M)j = (pM)r + s, \quad \text{where} \quad 0 \le r \;\; \text{and} \;\; 0 \le s < pM. $$
  If $0 \leq j < p$, then $s = Mj \, \mod{pM} = Mj$ and we can then solve to obtain $r = qMj$. Thus for $0 \leq j < p$, the formula
  \begin{equation}\tag{$\Div_{t_1,t_2}(j,r,s)$} t_2j = t_1r +s \; \land \; 0 \le r \; \land \;  0 \le s < t_{1} \end{equation}
is satisfied by the triple $(j, qMj, Mj)$. Furthermore, for such $j$ this formula cannot be satisfied by any other values of the second and third arguments.
\end{proof}

We now prove some additional capabilities of the parameters $t_1 = pM$, $t_2 = pqM^2 + M$ that will be required in order to transform the entire formula \eqref{eq:NP1} into a formula in $t_1$ and $t_2$ alone. 

\begin{lem} \label{lem:cong}
  The congruence relation modulo $M$ is definable using just the multiplications by $t_{1}$ and $t_{2}$.
\end{lem}

\begin{proof}
Let $\CongM_{t_1,t_2}(b,c,w_{1},w_{2})$ be the formula
$$b - c - t_1w_{1} - t_2w_{2} = 0.$$
  Since $\gcd(t_1,t_2) = M$, the condition $b \equiv c \, \mod{M}$ is expressed as:
  $$ \ex w_{1},w_{2} \quad \CongM_{t_1,t_2}(b,c,w_{1},w_{2}).$$
\end{proof}

\begin{lem} \label{lem:p} The constant $p$ is definable using just the multiplications by $t_{1}$ and $t_{2}$.
\end{lem}

\begin{proof}
 Since $t_2 / t_1 = qM + 1/p$, $p$ is the smallest positive integer $v$ such that $t_1 | t_2 v$. Since $t_2 p / t_1 = t_2 / M = pqM + 1$, we can express that a pair of variables $u,v$ satisfy $(u,v) = (pqM+1,p)$ by the formula  
 \[ u > 0 \;\land\; t_2v = t_1u \; \land \; \for v',u' \;\; 0 < v' < v \to t_2v' \neq t_1u' \]
which we denote by $\Equalp_{t_{1},t_{2}}(v,u)$.  

%  However, in order to avoid unnecessary alternation of quantifiers we can specify the value of $u$: if $t_2p = t_1u$, then $u=pqM+1$. Thus we can express the pair $(pqM+1,p)$ by the formula 

%$$ (**) \quad (pqM^{2} + M)v = (pM)u \; \land \; u > 0 \; \land \;  \for v',u' \;\; 0 < v' < v \to (pqM^{2}+M)v' \neq (pM)u'. $$
  
\end{proof}

\begin{lem} \label{lem:tfloor}
  Suppose $p$, $q$, and $M$ are positive integers such that $p/q \notin \Z$. If $t_1 = pM$ and $t_2 = pqM^2+M$ then $\floor{t_1^2 / t_2} = \floor{p/q}$.
\end{lem}

\begin{proof}
  First, we have
  \[ t_1^2 / t_2 = p^2M^2 / (pqM^2 + M)  = p / (q + 1/pM) < p/q, \]
  so $\floor{t_1^2 / t_2} \leq \floor{p/q}$. On the other hand, since $p/q \notin \Z$ we have:
\[p \ge \floor{p/q}q + 1 > \floor{p/q}q + \floor{p/q}/pM = \floor{p/q}(q + 1/pM).\]
This means $t_{1}^{2}/t_{2} = p/(q + 1/pM) > \floor{p/q}$, and thus $\floor{t_{1}^{2}/t_{2}} = \floor{p/q}$.
\end{proof}

%Note that multiplications by $p,q$ and $M$ appear separately in $S_{p,q,M}$.
%First, we can multiply by $M$ every inequality that involves multiplications in $p,q$.
%For example,
%$$py_{1} - qy_{2}  \ge px_{1} - qx_{2}  \ge 0 \quad \iff \quad pMy_{1} - qMy_{2} \; \ge \;  pMx_{1} - qMx_{2} \; \ge \; 0.$$
%So it is equivalent to use multiplications by $pM,qM,M$ instead of by $p,q,M$ in $S_{p,q,M}$. After doing so and also clearing denominators, we obtain that $S_{p,q,M}$ is the set of integers $z$ such that: 

\begin{proofof}{Theorem \ref{thm:2PPA}} In order to apply Proposition \ref{prop:GP}, we must first multiply by $M$ every inequality in \eqref{eq:NP1} that involves multiplication by $p$ or $q$. This works because multiplications by $p$, $q$, and $M$ appear separately in \eqref{eq:NP1}. After doing so and clearing some denominators, we obtain the equivalent formula
\begin{align} 
 \Phi'_{p,q,M}(z) \;\; = \quad  & \ex y_1, y_2: \nonumber \\
  & \: \: \: 0 < z \leq p/q \label{eq1} \\ 
  \land & \: \: \: y_2 \equiv z \mod{pM} \label{eq2} \\ 
  \land & \: \: \: p/q < y_2 + 1 \leq p \label{eq3} \\
  \land & \: \: \: qM y_2 < pM y_1 \label{eq4} \\
  \land & \: \: \: \for x_1,x_2 \;\;\;  
\lnot
\left\{
\begin{matrix}
pMy_{1} - qMy_{2} \; \ge \; pMx_{1} - qMx_{2} \; \ge \;  0\\
y_{2} \;>\; x_{2} \;>\; 0
\end{matrix}
\right\}. \label{eq5}
\end{align}
\medskip

\noindent
Here~\eqref{eq3} is equivalent to $\floor{p/q} \le y_{2} < p$ in \eqref{eq:NP1} because $y_{2} \in \Z$.  Now consider the formula 

  \begin{align}
\Psi_{t_{1},t_{2}}(z) \;\; = \quad    & \ex y_1, y_2, w_1, w_2, u, v, r, s: \nonumber \\
    & \: \: \: 0 < t_2 z \leq t_1^2 \tag{\ref{eq1}'}\label{eq1'} \\ 
  \land & \: \: \: \CongM_{t_1,t_2}(y_2,z,w_1,w_2) \tag{\ref{eq2}'}\label{eq2'} \\ 
  \land & \: \: \:  \Equalp_{t_1,t_2}(u,v) \land \: t_1^2 < t_2(y_2+1) \leq t_{2}v \tag{\ref{eq3}'} \label{eq3'} \\
  \land & \: \: \:  \Div_{t_1,t_2}(y_2,r,s) \land \: r < t_1y_1 \tag{\ref{eq4}'}\label{eq4'} \\
  \land & \: \: \:  \for x_1, x_2 \; \; \; \big(0 < x_{2} < y_{2} \; \land \; \Div_{t_{1},t_{2}}(x_{2},r',s')\big) \to  \lnot \big(0 \le t_{1}x_{1} - r' \le t_{1}y_{1} - r\big).
  \tag{\ref{eq5}'}\label{eq5'}
\end{align}

  \smallskip
  
\noindent
It only remains to show that $\Phi'_{p,q,M}(z)$ and $\Psi_{t_1,t_2}(z)$ are equivalent. We have:

\smallskip

$(\ref{eq1}) \leftrightarrow (\ref{eq1'})$ This follows by rounding down both equations to the nearest integer and applying Lemma \ref{lem:tfloor}.

$(\ref{eq2}) \leftrightarrow (\ref{eq2'})$ This is Lemma \ref{lem:cong}.

$(\ref{eq3}) \leftrightarrow (\ref{eq3'})$ We can again apply Lemma \ref{lem:tfloor} to replace $p/q$ in (\ref{eq3}) by $t_1^2 / t_2$, since every other quantity in \ref{eq3} is an integer. By Lemma \ref{lem:p}, the formula $\Equalp_{t_{1},t_{2}}(v,u)$ fixes the value of $v$ to be $p$, so we can now replace $p$ by $v$ to obtain \ref{eq3'}.

$(\ref{eq3}) \rightarrow [(\ref{eq4}) \leftrightarrow (\ref{eq4'})]$  By (\ref{eq3}), we have $0 \leq y_2 < p$, so by Proposition \ref{prop:GP}, the condition $\Div_{t_1,t_2}(y_2,r,s)$ fixes the value of $r$ to be $qMy_2$. Here we modify (\ref{eq4}) by replacing $qMy_2$ by $r$ and $pM y_1$ by $t_1y_1$ to obtain (\ref{eq4'}).

$(\ref{eq4}) \rightarrow [(\ref{eq5}) \leftrightarrow (\ref{eq5'})]$ Using (\ref{eq4'}) which we have already shown to be equivalent to (\ref{eq4}), we can replace $qMy_2$ by $r$. Using the definition of $t_1$, we can also replace $pMy_1$ by $ty_1$ and $pMx_1$ by $t_1x_1$. So $(\ref{eq5})$ is equivalent to
 \[ \for x_1,x_2 \;\;\;  
\lnot
\left\{
\begin{matrix}
ty_{1} - r \; \ge \; t_1x_{1} - qMx_{2} \; \ge \;  0\\
y_{2} \;>\; x_{2} \;>\; 0
\end{matrix}
\right\} ,\]
or in another form
\[ \for x_1,x_2 \;\;\; 0 < x_2 < y_2 \rightarrow \lnot [ty_{1} - r \; \ge \; t_1x_{1} - qMx_{2} \; \ge \;  0]. \label{eq:5intermediate}\]
Since the hypothesis $x_2 < y_2$ along with $y_2 < p$ from (\ref{eq3}) implies $x_2 < p$, we can (by Proposition \ref{prop:GP}) insert the condition $\Div_{t_{1},t_{2}}(x_{2},r',s')$ into the hypothesis to fix $r'$ equal to $qMx_2$. Accordingly substituting in $r'$ for $qMx_2$, we obtain (\ref{eq5'}).

So $\Phi_{p,q,M}, \Phi'_{p,q,M}$ and $\Psi_{t_{1},t_{2}}$ are all equivalent. This finishes the proof of Theorem~\ref{thm:2PPA}.
\end{proofof}

\begin{proofof}{corollaries~\ref{cor:qe_hard} and~\ref{cor:positive_hard}}
The formula $\Psi'_{t_{1},t_{2}}(z) \coloneqq (0 < z \le t_{1}^{2}/t_{2}) \, \land \, \lnot \Psi_{t_{1},t_{2}}(z)$ is satisfied only by those $z \in [\mu,\nu] \backslash \bigcup_{i=1}^{n}\AP_{i}$ (see~\eqref{eq:equal_intersection}).
This formula defines a 2-parametric family $S_{t_{1},t_{2}}$.
So the condition $|S_{t_{1},t_{2}}| > 0$, which is equivalent to $\textsc{AP-COVER}$, cannot be expressed using polynomial-time relations in $t_{1}$ and $t_{2}$.
Similarly, any expansion of parametric PA with polynomial-time predicates cannot have full quantifier elimination. For otherwise we can apply it to the sentence $\ex z \; \Psi'_{t_{1},t_{2}}(z)$ and get an equivalent Boolean combination of polynomial-time relations in $t_{1},t_{2}$.
\end{proofof}

\bigskip

\section{Counting-universality of $2$-parametric Presburger formulas}

Consider a $k$-parametric PA formula:
\begin{equation}\label{eq:k-param}
\Phi_{\u}(\x) = Q_{1}y_{1} \; Q_{2}y_{2} \; \dots Q_{m}y_{m} \;\; \Theta_{\u}(\x, \y).
\end{equation}
Here $\u \in \Z^{k}$ are the $k$ scalar parameters, $\x \in \Z^{d}$ are the free variables, $\y = (y_{1},\dots,y_{m}) \in \Z^{m}$ are the quantified variables, $Q_{1}, \dots Q_{m} \in \{\for,\ex\}$ are the quantifiers, and $\Theta_{\u}(\x,\y)$ is a Boolean combination of linear inequalities in $\x,\y$ with coefficients and constants from $\Z[\u]$.
This formula defines a parametric family $S_{\u}$.
%% the form:
%% $$
%% \sum_{i=1}^{n} \al_{i}(\u) x_{i} + \sum_{i=1}^{m} \be_{i}(\u) y_{i} \; \le \; \gam_(\u),
%% %%\qquad i = 1, \dots, a
%% $$
%% with $\al_{i},\be_{i},\gam \in \Z[\u]$.
%% We are not losing any generality by restricting the variables $\x,\y$ and to $\N$ instead of $\Z$.
%% This is because any usual integer $z \in \Z$ can be written as $z_{1}-z_{2}$ with $z_{1},z_{2} \in \N$.

\begin{definition}\label{def:counting_reduce}
  We say that a $k_{1}$-parametric family $S_{\u}$ \emph{counting-reduces} to an $k_{2}$-parametric family $S'_{\t}$ if there exists $f = (f_{1},\dots,f_{k_{2}}) : \Z^{k_{1}} \to \Z^{k_{2}}$ with $f_{i} \in \Z[\u]$ such that for every $\u \in \Z^{k_{1}}$ we have:
  $$|S_{\u}| = \infty \; \Rightarrow \; |S'_{f(\u)}| = \infty \quad\text{and}\quad |S_{\u}| < \infty \Rightarrow |S_{\u}| = |S'_{f(\u)}|.$$
\end{definition}

\begin{thm}\label{th:universal}
Every $k$-parametric PA family $S_{\u}$ counting-reduces to another $2$-parametric PA family $F_{s,t}$ with the same number of alternations. In other words, $2$-parametric PA families are counting-universal.
\end{thm}

%By this theorem, the reduction from three to two parameters in Section~\ref{sec:main_proof} is just a neat special case.
First we prove the following lemma.

\begin{lem}\label{lem:bounded}
%%Every $k$-parametric PA formula $S_{\u}$ counting-reduces to another formula $S'_{\u}$ with polynomially bounded variables.
For every formula $\Phi_{\u}$ of the form~\eqref{eq:k-param}, 
  there exist $\mu,\mu',\nu_{1},\dots,\nu_{m} \in \Z[\u]$
such that for every value $\u \in \Z^{k}$ we have:
\begin{itemize}
\item[i)] $|S_{\u}| = \infty$ if and only if:
  $$
\ex\ts\x \;\; \bigg[\mu(\u) \le \norm{\x}_{\infty} \le \mu'(\u) \; \land \; Q_{1} \big(|y_{1}| \le \nu_{1}(\u)\big)  \; \dots \; Q_{m}\big( |y_{m}| \le \nu_{m}(\u)\big) \;\; \Theta_{\u}(\x, \y) \bigg]
  $$
\item[ii)] If $|S_{\u}| < \infty$ then for every $\x \in \Z^{d}$:
$$
S_{\u}(\x) = \textup{true} \iff \norm{\x}_{\infty} \le \mu(\u) \; \land \; Q_{1} \big( |y_{1}| \le \nu_{1}(\u) \big)  \; \dots \; Q_{m} \big( |y_{m}| \le \nu_{m}(\u) \big) \;\; \Theta_{\u}(\x, \y).
$$
\end{itemize}
Here $\norm{\cdot}_{\infty}$ is the $\ell_{\infty}$--\. norm.
So $\mu(\u) \le \norm{\x}_{\infty}$ stands for $\bigvee_{i=1}^d \big(x_{i} \le -\mu(\u)  \; \lor \;  \mu(\u) \le x_i \big)$ and  $\norm{\x}_{\infty} \le \mu'(\u)$ stands for $\bigwedge_{i=1}^d \big( -\mu'(\u) \le x_{i} \le \mu'(\u) \big)$.
Each restricted quantifier $Q_{i}\big( |y_{i}| \le \nu_{i}(\u) \big)$  means exits/for all $\,y_{i}$ in the interval $[-\nu_{i}(\u),\nu_{i}(\u)]$.\footnote{Here we understand that $\mu,\mu',\nu_{i}$ have positive values for all $\u \in \Z^{k}$.}
\end{lem}

\begin{proof}
Consider a usual,  non-parametric PA formula:
$$
\Phi(\x) \; = \; Q_{1}y_{1} \; Q_{2}y_{2} \; \dots Q_{m}y_{m} \;\; \Theta(\x, \y),\quad \x \in \Z^{n},
$$
which defines some set $S \subseteq \Z^{n}$.
Recall Cooper's quantifier elimination procedure for Presburger arithmetic (see~\cite{Oppen}).
Applying it to $\Phi(\x)$, we obtain an \emph{equivalent} quantifier free formula $\Phi'(\x)$, which may contain some extra divisibility predicates.
By Theorem~2 of~\cite{Oppen}, after eliminating all $m$ quantifiers from $\Phi$, we obtain the following bounds:
$$
c' \;\le\; c^{4^{m}}, \quad s' \;\le\; s^{(4c)^{4^{m}}}, \quad a' \;\le\; a^{4^{m}} s^{(4c)^{4^{m}}},
$$
where:
\begin{itemize}
\item $c$ is the number of distinct integers that appeared as coefficients or divisors in $\Phi$,
\item $s$ is the largest absolute value of all integers that appeared in $\Phi$ (coefficients + divisors + constants),
\item $a$ is the total number of atomic formulas in $\Phi$ (inequalities + divisibilities),
\end{itemize}
and $c',s',a'$ are the corresponding quantities for $\Phi'$.
Now assume $c,m$ and $n$ are fixed.
Then we have:
$$
c'  \;\le\; \const, \quad s' \;\le\; s^{\const}, \quad a'  \;\le\; a^{\const} s^{\const},
$$
where $\const = \const(c,m)$ is fixed.
So in this case  $\Phi'$ has at most a fixed  number of coefficients and divisors.
%%Also the largest integer in $S'$ are bounded polynomially by the largest integer in $S$.

Denote by $D$ the common multiple of all divisors in $\Phi'$.
We have $D \le s^{\const}$.
Let $\L = \langle D e_{1}, \dots, D e_{n} \rangle$ be the lattice of $\Z^{n}$ consisting of $\x \in \Z^{n}$ whose coordinates are all divisible by $D$.
Fix some particular coset  $\C$ of $\L$ and restrict $\x$ to $\C$.
Then in $\Phi'(\x)$, all divisor predicates have fixed values (either true or false) as $\x$ varies over $\C$.
So over $\C$, the formula $\Phi'(\x)$ is just a Boolean combination of linear inequalities in $\x$, which represents a \emph{disjoint union} of some rational polyhedra in $\R^{n}$.
Each such polyhedron $P$ can be described by a system of \emph{fixed} length, because there are only at most $c'$ different coefficients for the $\x$ variables.
The integers in the system are also bounded by $s^{\const}$.
We consider $P \cap \C$. By the fundamental theorem of Integer Programming\footnote{We are rescaling $\L$ to $\Z$ before applying this bound.} (see~\cite[Th.~16.4 and Th.~7.1]{Schrijver}), we have:
$$P \cap \C \;=\; \conv(\ov v_{1},\dots,\ov v_{p}) \,+\, \Z_{+}\langle \ov w_{1},\dots,\ov w_{q} \rangle$$
for some $\ov v_{i}, \ov w_{j} \in \Z^{n}$ with  $\norm{\ov v_{i}}_{\infty}, \norm{\ov w_{j}}_{\infty} < s^{\const'}$.
Here $\const' = \const'(c,m,n)$ is fixed.
From this, it is easy to see that there is $\const'' = \const''(c,m,n)$ such that for every polyhedron $P$ in the disjoint union, we have:
$$
\gathered
|P \cap \C| = \infty \quad \iff \quad \text{there is } \x \in P \cap \C \;\;\text{with}\;\; s^{\const''} < \norm{\x}_{\infty} < s^{2\const''},
\\
|P \cap \C| < \infty \quad \Longrightarrow \quad P \cap \C  \,\subseteq\, [-s^{\const''}, s^{\const''}]^{n}.
\endgathered
$$

Since this holds for every coset $\C$ of $\L$, we conclude that there is $\const_{0} = \const_{0}(c,m,n)$ such that:
\begin{gather}
|S|  = \infty \quad \iff \quad \ex\, \x \;\;\text{with}\;\; s^{\const_{0}} < \norm{\x}_{\infty} < s^{2\const_{0}} \;\text{ and }\; \Phi'(\x)=\text{true} \label{eq:case1}
\\
|S|  < \infty \quad \Longrightarrow \quad \for\,\x  \;\; \left( \Phi'(\x) = \text{true} \;\to\;  \norm{\x}_{\infty} \le s^{\const_{0}} \right). \label{eq:case2}
\end{gather}

This gives us a bound for $\x$.
%%Now assume $|S|<\infty$ and $S(\x) = \text{true}$
Now for every $\x$ with $\norm{\x}_{\infty} \le s^{\const_{0}}$,
by the same argument, it is enough to decide the (substituted) sentence $\Phi(\x)$ over those $y_{1}$ with $|y_{1}| \le s^{\const_{1}}$.
In other words, for every such value for $\x$, we may replace $Q_{1}y_{1}$ by $Q_{1}\big(|y_{1}| \le s^{\const_{1}} \big)$ in $\Phi(\x)$ to obtain a new formula $\Phi_1(\x)$, which is equivalent to the original formula $\Phi(\x)$.
Working inwards, we can likewise bound $|y_{2}|$ by $s^{\const_{2}}$, $|y_{3}|$ by $s^{\const_{3}}$, etc.
Therefore, in case $|S| < \infty$, the whole formula $\Phi$ is equivalent to one with bounded quantifiers on all $y_{i}$.
Also by~\eqref{eq:case1}, we have $|S| = \infty$ if and only if some $ s^{\const_{0}} < \norm{\x}_{\infty} < s^{2\const_{0}}$ satisfies it.
For $\x$ in this range, we can again bound $y_{1} ,y_{2},$ etc., accordingly by some other powers of $s$.
Note that we can bound each $y_{i}$ by a common larger power of $s$ for both cases~\eqref{eq:case1} and~\eqref{eq:case2}.

In a $k$-parametric PA formula $\Phi_{\u}(\x)$, we consider $m,n$ and $c$ to be fixed.
Since all coefficients and constants of $\Phi_{\u}$ are in $\Z[\u]$, we can bound $s$ by some polynomial in $\u$.
Thus, every $s^{\const}$ is also bounded by some polynomial in $\u$.
This proves Lemma~\ref{lem:bounded}.
%This follows from~\cite{Goodrick}. Can also be directly derived from~\cite{Oppen}.
%Intuitively, its says that $|S_{\u}| = \infty$ if and only if some large enough $\x$ satisfies $S_{\u}$.
%Otherwise, all $\x$ that satisfy $S_{\u}$ must be small enough.
\end{proof}

\begin{rem}
  In the above application of Cooper's elimination, if only $m,n$ are fixed but not $c$, then we no longer have the bound $s' \le s^{\const}$.
  Instead, we would have $c',\log s' \le \text{poly}(c, \log s)$.
  A bound of this type is important for showing that the decision problem for classical PA with a bounded number of variables falls within the Polynomial Hierarchy (see e.g.~\cite{Gradel}).
However, it would not be strong enough for our argument, which crucially needs $\log s' = O(\log s)$.
\end{rem}

From Lemma \ref{lem:bounded}, it is easy to see that $S_{\u}$ counting-reduces to the family $\wt S_{\u}$ defined by the following formula $\wt \Phi_\u(\x, \wt x)$:
$$
\aligned
\wt \Phi_{\u}(\x,\wt x) \;\; = \quad &\Big[ \wt x \ge 0 \; \land \;  Q_{1} \big( |y_{1}| \le \nu_{1}(\u) \big)  \; \dots \; Q_{m}  \big( |y_{m}| \le \nu_{m}(\u) \big) \;\; \mu(\u) \le \norm{\x}_{\infty} \le \mu'(\u) \; \land \; \Theta_{\u}(\x, \y) \Big] \; \lor  \\
 &\Big[ \wt x = 0 \; \land \;  Q_{1} \big( |y_{1}| \le \nu_{1}(\u) \big) \; \dots \; Q_{m} \big( |y_{m}| \le \nu_{m}(\u) \big) \;\; \norm{\x}_{\infty} \le \mu(\u) \; \land \;  \Theta_{\u}(\x, \y) \Big].
\endaligned
$$
%%Here we are not writing $\u$ for brevity.
Here the bounds on $\norm{\x}_{\infty}$ are moved to after the quantifiers on $y_{i}$ without changing the meaning.
The dummy variable $\wt x$ is used to make sure that $|\wt S_{\u}| = \infty$ in the first case.

\begin{proof}[Proof of Theorem~\ref{th:universal}]
  We show that $\wt S_{\u}$ counting-reduces to a $2$-parameter family $F_{s,t}$, defined by a new formula $\Psi_{s,t}$.
  First, we list all the different scalar terms that appear in $\wt \Phi_{\u}$, either as coefficients or constants (including all $\mu,\mu',\nu_{i}$), as $\del_{0}(\u),\dots,\del_{r}(\u)$.
  Now suppose we need to multiply some $z \in \N$ by $\del_{0}(\u),\dots,\del_{r}(\u)$ and also know that
\begin{equation}\label{eq:t_bound}
  -t/2 < \del_{0}(\u)\ts z,\,\dots,\,\del_{r}(\u)\ts z < t/2
\end{equation}
for some $t \in \Z$.
The following base-$t$ concatenation, which is similar to~\eqref{eq:new_params}, can be used.
Essentially, we encode the ``multi''-product $(\del_{0}(\u)\ts z, \dots, \del_{r}(\u)\ts z)$ as a single product:
\begin{equation*}
\del_{0}(\u)\ts z \; + \; t\, \del_{1}(\u)\ts z \; + \; \dots \; + t^{r}\, \del_{r}(\u)\ts z \; = \; (\del_{0}(\u) + t\, \del_{1}(\u) + \dots + t^{r} \del_{r}(\u))\, z.
\end{equation*}
In other words, if $s = \del_{0}(\u) + t\, \del_{1}(\u) + \dots + t^{r} \del_{r}(\u)$ and:
\begin{equation}\label{eq:ts_div}\tag{$\Div_{s,t}(z,z_{0},\dots,z_{r})$}
s\,z \;=\; z_{0} + t\,z_{1} + \dots + t^{r}z_{r} \;\; \land \;\;  t/2 < z_{0},\dots,z_{r} < -t/2,
\end{equation}
then we must have $z_{0} = \del_{0}(\u)\ts z,\, \dots,\, z_{r} = \del_{r}(\u)\ts z$.
Indeed, by subtracting we get $z_{0} - \del_{0}(\u)\ts z \equiv 0 \mod t$, which implies $z_{0} = \del_{0}(\u)\ts z$ because $-t/2 < z_{0},\,\del_{0}(\u)\ts z < t/2$.
The same argument applies to other $z_{i}$.
%Denote the conditions~\eqref{eq:ts_div} by $\Div_{s,t}(z,z_{0},\dots,z_{r})$.

\medskip

Observe that in $\wt \Phi_{\u}$, all variables $\x$ and $\y$ are bounded by polynomials in $\u$.
Hence, we can pick $\eta(\u) \in \Z[\u]$ so that for every value $\u \in \Z^{k}$, the condition~\eqref{eq:t_bound} is always satisfied when $t = \eta(\u)$ and $z$ is either the constant $1$ or any of the possible values of the $\x,\y$ variables.
Our reduction map $f : \Z^{k} \to \Z^{2}$ can now be defined by letting 
\begin{equation*}
t = \eta(\u); \hspace{0.2in} s = \del_{0}(\u) + t\, \del_{1}(\u) + \dots + t^{r} \del_{r}(\u).
\end{equation*}

\medskip

Now we can define $\Psi_{s,t}(\x,\wt x)$ from $\wt \Phi_{\u}(\x,\wt x)$.
%%First let $F_{s,t}(\x,\wt x) = \wt S_{\u}(\x,\wt x)$.
%%There are $n$ free variables $x_{1},\dots,x_{n}$ (also $\wt x$) and $m$ quantified variables $y_{1},\dots,y_{m}$ in $\wt S_{\u}$.
We need $(m+d+1)(r+1)$ extra variables:
$$
\w = (w_{ij})_{1 \le i \le d,\, 0 \le j  \le r},\; \w' = (w'_{ij})_{1 \le i \le m,\, 0 \le j \le r} \quad \text{and} \quad \v = (v_{j})_{0 \le j \le r}.
$$
%%Consider the first clause $\big(\dots Q_{m}[0 \le y_{m} \nu_{m}(\u)]\big\)$
Assuming the last quantifier $Q_{m}$ in $\wt \Phi_{\u}$ is $\ex$, we insert 
$$
(\star) \qquad  \ex\, \w, \w', \v \; \bigwedge_{i=1}^{d} \Div_{s,t}(x_{i},w_{i0},\dots,w_{ir}) \; \land \; 
\bigwedge_{i=1}^{m} \Div_{s,t}(y_{i},w'_{i0},\dots,w'_{ir}) \; \land \;
\Div_{s,t}(1,v_{0},\dots,v_{r})
$$
right before $\Theta_{\u}(\x,\y)$, i.e., replace $\Theta_{\u}(\x,\y)$ by $(\star) \land \Theta_{\u}(\x,\y)$.
Then in $\wt \Phi_{\u}$ we replace every term $\del_{j}(\u)\, x_{i}$ by $w_{ij}$, every term $\del_{j}(\u)\ts y_{i}$ by $w'_{ij}$ and every term $\del_{j}(\u)$ by $v_{j}$.
%%\footnote{This is done for each of the  two appearances of $\Theta_{\u}(\x,\y)$ in $\wt S_{\u}$.}
Now $\wt \Phi_{\u}$ becomes the desired $\Psi_{s,t}$.
In case $Q_{m} = \for$, we insert:
$$
(\star\star) \qquad  \for\, \w, \w', \v \; \bigvee_{i=1}^{d} \lnot\Div_{s,t}(x_{i},w_{i0},\dots,w_{ir}) \; \lor \; 
\bigvee_{i=1}^{m} \lnot\Div_{s,t}(y_{i},w'_{i0},\dots,w'_{ir}) \; \lor \;
\lnot\Div_{s,t}(1,v_{0},\dots,v_{r})
$$
right before $\Theta_{\u}(\x,\y)$, i.e., replace $\Theta_{\u}(\x,\y)$ by $(\star\star)\lor\Theta_{\u}(\x,\y)$.
Again, replace every term $\del_{j}(\u)\, x_{i}$ by $w_{ij}$, every term $\del_{j}(\u)\ts y_{i}$ by $w'_{ij}$ and every term $\del_{j}(\u)$ by $v_{j}$.
This gives $\Psi_{s,t}$.

Note that $\Psi_{s,t}$ still has the form $\big[\dots\big] \lor \big[ \dots \big]$ with each disjunct containing $m$ alternations $Q_{1} \dots Q_{m}$.
This formula is equivalent to a formula in prenex normal form with $m$ quantifier alternations, so we are done.
\end{proof}

\begin{rem}
  In case $S_{\u}$ is defined by a quantifier-free formula, i.e., $m=0$, we only need to insert $(\star)$, without the $\ex$ quantifiers, before  $\Theta_{\u}(\x,\y)$.
  This is because $\Div_{s,t}(z,z_{0},\dots,z_{r})$ uniquely determines $z_{0},\dots,z_{r}$ in $z$.
  So in this case $S_{\u}$ also counting-reduces to a quantifier-free $F_{s,t}$, although the latter has many more free variables.
Thus, the study of integer point counting functions on $k$-parametric polyhedra reduces to the case of $2$-parametric polyhedra in higher dimensions.
\end{rem}

\bigskip

\section{Counting points in parametric unordered Presburger families in polynomial time}

In this section, we consider the reduct of multi-parametric Presburger arithmetic to the language without ordering, so that basic quantifier-free formulas are equivalent to Boolean combinations of equations of the form $f_1(\t) x_1 + \ldots + f_n(\t) = g(\t)$, where $\t = (t_1, \ldots, t_k)$ is a tuple of parameters and $f_1, \ldots, f_m, g \in \Z[\t]$. As always, we are allowed to quantify over the variables $x_i$ but not over the parameters $\t$. Note that if there is no parameter $\t$, this would correspond to studying the first-order logic of the additive group $(\Z; +)$.
More precisely:

\begin{definition}
\label{def:unordered}
 A \emph{$k$-parametric unordered PA family} is a collection $$\{S_\bft : \bft = (t_1, \dots, t_k) \in \Z^k \}$$ of subsets of $\Z^d$ which can be defined by an equation of the form 
 \begin{equation}\label{eq:unordered_Presburger}
S_\bft = \{ \x \in \Z^d \;\; : \; \;Q_{1}y_{1} \; Q_{2}y_{2} \; \dots Q_{m}y_{m} \;\; \Theta_\bft(\x, \y) \},
\end{equation}
where the $Q_i \in \{\forall, \exists\}$ are quantifiers for variables $y_i$ ranging over $\Z$ and $\Theta_\bft(\x, \y)$ is a Boolean combination of linear equations with coefficients in $\Z[\bft]$.

\end{definition}

%\begin{rem}
%For technical reasons, it is convenient to allow the parameters $\t$ to range over $\Z^k$ and not just $\N^k$ in this section in order to easily apply results from \cite{vdDH}. But the main result of this section would also be true if we restricted the values of $\t$ to $\N^k$.
%\end{rem}

For example, $$\left(x_1 = 0\right) \wedge \exists x_2 \exists x_3 \left( x_2 t_1 + x_3 t_2 = 1 \right)$$ defines a $2$-parametric unordered PA family $\{S_\t \subseteq \Z : \t \in \Z^2\}$ such that $S_\t = \{0\}$ if $\gcd(t_1, t_2) = 1$ and $S_\t = \varnothing$ otherwise.

\begin{thm}
\label{th:unordered}
Suppose that $S_{\t} \subseteq \Z_d$ is a $k$-parametric unordered PA family. Then we have:

\begin{enumerate}
\item There is a polynomial-time algorithm to decide whether $S_\t$ is nonempty.
\item There is a polynomial-time algorithm on input $\t$ which decides whether or not $S_\t$ is finite or infinite.
\item There is a polynomial-time evaluable function $g: \Z^k \rightarrow \N$ such that whenever $S_\t$ is finite, $g(\t) = |S_\t|$.
\end{enumerate}
\end{thm}

In fact, the proof of Theorem~\ref{th:unordered} will show that the decision algorithms for (1) and (2) rely upon only a few basic, concrete number-theoretic operations on $\t$, such as $\gcd$ and a couple of related functions.

To prove Theorem~\ref{th:unordered}, we need to recall some notation from \cite{vdDH}. To eliminate quantifiers, they work in a two-sorted language $L_2$ in which variables $x_i$ and parameters in $\t$ are assigned to objects of distinct domains, called the \emph{group sort} and the \emph{ring sort}, respectively. \textbf{For our purposes, the group sort and the ring sort are two disjoint copies of $\Z$.} The variables $x_i$ and $y_i$ will always range over values in the group sort, and the parameters $t_i$ will always range over values in the scalar sort. In other words, we can think of the parameters $t_1, \ldots, t_k$ as ``typed variables'' ranging over a domain of possible parameter values in the scalar sort (a copy of $\Z$), and $x_1, x_2, \ldots$ as variables of a distinct type ranging over values in the group sort (which is a different copy of $\Z$), and the parameters $t_i$ act upon the group sort by scalar multiplication.

The language $L_2$ consists of the following nonlogical symbols (in addition to equality):

\begin{itemize}
\item Within the scalar sort, constant symbols for $0$ and $1$, a unary operation $-$ for negation, ring operations $+$ and $\cdot$, and four additional binary operations $g, \alpha, \beta,$ and $\gamma$ (whose interpretation is explained below);
\item Within the group sort, a constant symbol for $0$, a unary operation $-$ for negation, and a symbol $+$ for addition;
\item A binary operation $\cdot$ such that $s \cdot x$ is a value in the group sort whenever $s$ is a value in the scalar sort and $x$ is a value in the group sort, denoting multiplication by $s$ in the usual sense; and
\item A binary relation symbol $|$ to be interpreted such that whenever $s$ is in the scalar sort and $x$ is in the group sort, $$s | x \Leftrightarrow \exists y \left( s \cdot y = x \right).$$
\end{itemize}
The binary operations $g, \alpha, \beta,$ and $\gamma$ between values in the scalar sort are interpreted so that $g(r,s) = gcd(r,s)$ and the following axioms hold for all values $r, s$ in the scalar sort:

$$r = \gamma(r,s) \cdot g(r,s),$$
$$1 = \alpha(r,s) \cdot \gamma(r,s) + \beta(r,s) \cdot \gamma(s,r).$$

We will use the following fact, proved in \cite{vdDH}:

\begin{thm}
\label{th:vdDH}
Any formula $\varphi_\t(\overline{x})$ in $k$-parametric unordered Presburger arithmetic is logically equivalent to a quantifier-free $L_2$-formula $\psi(\overline{x}, \t)$: that is, with the natural interpretations of the symbols from $L_2$ given above, $$\models \forall \overline{x} \in \Z^d \,  \forall \t \in \Z^k \, \left( \varphi_\t(\overline{x}) \leftrightarrow \psi(\overline{x}, \t) \right),$$ where $\psi(\overline{x}, \t)$ is a Boolean combination of equations $s_1(\overline{x}, \t) = s_2(\overline{x}, \t)$ and divisibility relations $s_3(\t) | s_1(\overline{x}, \t)$, where $s_1(\overline{x}, \t)$, $s_2(\overline{x}, \t)$, and $s_3(\t)$ are $L_2$-terms, i.e. expressions built up using only the operations in $L_2$ and the displayed parameters and variables.
\end{thm}

\textit{Proof of Theorem~\ref{th:unordered}:} Say $\varphi_\t(\overline{x})$ defines a $k$-parametric unordered PA family in $\Z^d$.

Note that (1) follows almost immediately from quantifier elimination: by Theorem~\ref{th:vdDH}, the formula $\exists \overline{x} \varphi_t(\overline{x})$ is equivalent to a quantifier-free $L_2$-formula $\psi(\t)$ in only the scalar sort of $\t$, which is a Boolean combination of equations and divisibility relations $|$ in the $k$ parameters using ring operations and the functions $g, \alpha, \beta,$ and $\gamma$, but all of these operations are polynomial-time computable.

For (2), let us assume (by Theorem~\ref{th:vdDH}) that $\varphi_\t(\overline{x})$ is a quantifier-free $L_2$-formula, and that $\varphi_t(\overline{x})$ is in disjunctive normal form: $$\varphi_t(\overline{x}) = \bigvee_{i=1}^m \theta_i(\overline{x}, \t),$$ where each $\theta_i(\overline{x}, \t)$ is a conjunction of \emph{literals}.\footnote{A literal is an \emph{atomic} $L_2$-formula, i.e. one containing no logical operations $\wedge, \vee$ or $\neg$, or the negation of an atomic formula.}

\begin{Claim}
For any fixed value of $\t \in \Z^k$ and of $i \in \{ 1, \ldots, m\}$, if $S_i := \{\overline{x} \in \Z^d : \, \models \theta_i(\overline{x}, \t) \}$, then $|S_i|$ is either $0$, $1$, or $\infty$.
\end{Claim}

\begin{proof}
By rearranging terms, we may assume that all atomic $L_2$-formulas in $\theta_i(\overline{x}, \t)$ have the form
\begin{equation}\tag{A}
r \, | \, s(\overline{x}, \t)
\end{equation}
or
\begin{equation}\tag{B}
s(\overline{x}, \t) = 0,
\end{equation}
where $s(\overline{x}, \t) = r_0 + \sum_{i=1}^d r_i \cdot x_i$ and $r_0, r_1, \ldots, r_n,$ and $r$ are terms in the scalar sort. The terms $r$ and $r_i$ may involve the parameters $\t$ and the operations $g, \alpha, \beta, \gamma$, but the details of this are irrelevant since $\t$ has a fixed value.

Write $$\theta_i(\overline{x}, \t) = \theta_A(\overline{x}, t) \wedge \theta_B(\overline{x}, \t)$$ where $\theta_A(\overline{x}, \t) $ is the conjunctions of all literals of type (A) and $\theta_B(\overline{x}, t) $ is the conjunction of all literals of type (B).

First we consider the atomic formulas of type (A). Each one defines some coset of a finite-index subgroup of $\Z^d$, and so the negation of such a formula defines a finite union of cosets of finite-index subgroups. Since the intersection of finitely many finite-index subgroups is of finite index, there is a single subgroup $H \leq \Z^d$ such that $[\Z^d : H ] < \infty$ and $\theta_A(\overline{x}, \t)$ defines a Boolean combination of cosets of $H$.

Now consider the atomic formulas of type (B). We decompose $\theta_B(\overline{x}, \t)$ further as $$\theta_B(\overline{x}, \t) = \theta^+_B(\overline{x}, \t) \wedge \theta^-_B(\overline{x}, \t)$$ where $\theta^+_B(\overline{x}, \t)$ is the conjunction of all positive (non-negated) atomic formulas of type (B) and $\theta^-_B(\overline{x}, \t)$ is the conjunction of all negative literals of type (B). Note that the set of solutions to $\theta^+_B(\overline{x}, \t)$ is of the form $(\vec{v} + S) \cap \Z^d$ where $S$ is a vector subspace of $\R^d$ and $\vec{v} \in \Z^d$.

Finally, suppose that there are at least two distinct elements $\overline{x}_1, \overline{x}_2 \in \Z^d$ in $S_i$, and to finish the proof of the Claim we will show that $S_i$ has infinitely many elements. In particular, both $\overline{x}_1$ and $\overline{x}_2$ are solutions to $\theta_A(\overline{x}, \t)$, so there are cosets $C_1, C_2$ of $H$ such that $\overline{x}_1 \in C_1$, $\overline{x}_2 \in C_2$, and any element $\overline{x} \in C_1 \cup C_2$ satisfies $\theta_A(\overline{x}, \t)$. Let $L \subseteq \R^d$ be the line passing through $\overline{x}_1$ and $\overline{x}_2$, and observe that since $\overline{x}_1$ and $ \overline{x}_2$ satisfy  $\theta^+_B(\overline{x}, \t)$ (which defines the intersection of an affine subspace with $\Z^d$), any other element of $L \cap \Z^d$ will also satisfy $\theta^+_B(\overline{x}, \t)$.

For any $j \in \Z$, let $\overline{x}(j) :=  \overline{x}_1 + j \cdot (\overline{x}_2 - \overline{x}_1)$ and $$X := \{j \in \Z : \overline{x}(j) \textup{ satisfies } \theta_i(\overline{x}, \t) \}.$$ Since $H$ is a finite-index subgroup of $\Z^d$, adding successive copies of the element $(\overline{x}_2 - \overline{x}_1)$ to $\overline{x}_1$ causes the $\overline{x}(j)$ to cycle through cosets of $H$, and the set of $j$ for which $\theta_A(\overline{x}(j), \t)$ is true is infinite (and periodic). As observed in the previous paragraph, \emph{every} $\overline{x}(j)$ lies on the line $L$, and hence $\theta^+_B(\overline{x}(j), \t)$ is always true, and we need only worry about the truth of $\theta^-_B(\overline{x}(j), \t)$. Now  $\theta^-_B(\overline{x}(j), \t)$ is true whenever $\overline{x}(j)$ \emph{avoids} every one of a finite number of affine subspaces $A_1, \ldots, A_\ell$ of $\R^d$, but given that $L$ is a line which contains some points satisfying the formula $\theta^-_B(\overline{x}, \t)$, each $A_i$ can only intersect $L$ in at most one point. Therefore $X$ is infinite, as we wanted.

\end{proof}

The Claim shows that we can define the set of values of the parameter $\t$ for which any given $\theta_i(\overline{x}, \t)$ has infinitely many solutions (for $\overline{x}$) by the formula $$\exists \overline{x}_1 \exists \overline{x}_2 \left(\overline{x}_1 \neq \overline{x}_2 \wedge \theta_i(\overline{x}_1, \t) \wedge \theta_i(\overline{x}_2, \t) \right),$$ and as before this is equivalent to a quantifier-free $L_2$-formula $\psi_i(\t)$ whose truth can be decided by a polynomial-time algorithm in $\t$. Finally, our original formula $\bigvee_{i=1}^m \theta_i(\overline{x}, \t)$ has infinitely many solutions just in case any one of the formulas $\theta_i(\overline{x}, \t)$ does, establishing (2).

By the argument above, for any $k$-parametric unordered PA family $S_\t$, there is a finite partition $\Z^k = X_1 \cup \ldots \cup X_\ell$ which is definable by quantifier-free $L_2$-formulas in $\t$ and such that $|S_\t|$ is constant as $\t$ varies over any of the sets $X_i$. Since deciding whether $\t \in X_i$ is polynomial-time decidable, this establishes (3). $\square$ (Theorem~\ref{th:unordered})

\bigskip

\section{Summary of Complexity Results}

To conclude, we summarize the complexity results which suggest that Theorem~\ref{thm:2PPA} may be the best we could hope for: weakening or changing various assumptions results in problems which can be resolved in polynomial time, or else (with unrestricted multiplication) have no algorithmic solutions at all.

Recall that Theorem~\ref{thm:2PPA} states that, if $\poly \neq \NP$, then there is a $\Sigma_2$ PA family $S_\t$ with two parameters $\t = (t_1, t_2)$ such that $|S_\t|$ cannot be computed in polynomial time given $\t$ as input.

However:
\medskip

$\bullet$ If we allow only a single parameter $t \in \N$ (or $\t \in \Z$), then for any PA family $S_t$, we can compute $|S_t|$ in polynomial time, even if $S_t$ has complexity $\Sigma_2$ or higher, by Corollary \ref{cor:1_param}.

\medskip

$\bullet$ If $S_\t$ is a $k$-parametric PA family defined by a formula of complexity $\Pi_1$ or $\Sigma_1$, then \cite{BW03} implies that there is a polynomial time algorithm to evaluate $|S_\t|$, for any finite number $k$ of parameters. If $S_\t$ is defined by a quantifier-free formula, then a polynomial-time algorithm was earlier given in \cite{barvinok94}.

\medskip

$\bullet$ If $S_\t$ is any $k$-parametric PA family defined by a formula with no inequalities (only equations), as in Section 4, then $|S_\t|$ can be evaluated in polynomial time, regardless of the number of quantifier alternations in the defining formula or the number of parameters.

\medskip
 
$\bullet$ In $k$-parametric PA formulas, we allow a restricted version of multiplication: the non-quantified parameters in $\t$ can be multiplied by terms containing the variables $\x$ and $\y$, but no multiplication between the $\x$ and $\y$ variables is allowed. Permitting unrestricted multiplication amongst the $\x$ and $\y$ variables in a parametric PA formula would obviously be bad, since the full first-order theory of $(\N, +, \cdot)$ is undecidable (by theorems of Church and Turing -- see, e.g., \cite{Church}). In fact, the Matiyasevich-Robinson-Davis-Putnam theorem \cite{Davis73} states that there is a \emph{single} multivariate polynomial $p(t, x_1, \ldots, x_d)$ such that if $\Phi_t(x_1, \ldots, x_d)$ is the formula expressing $$p(t, x_1, \ldots, x_d) = 0,$$ then the set of $t \in \N$ for which $ \Phi_t(x_1, \ldots, x_d)$ defines a nonempty subset of $\Z^d$ is not computable (much less in polynomial time). Note that here we have only a single parameter $t$, no quantifiers in the formula $\Phi_t$, and mere equations rather than inequalities.

\medskip

$\bullet$ On the other hand, if we allow \emph{no multiplication}, even by parameters (cf. Example \ref{ex:c}), then $\abs{S_{\t}}$ will be computable in polynomial time; in fact, it has a nice form as a piecewise-defined quasi-polynomial \cite{Woods15}.

\subsection*{Acknowledgements}
We thank Igor Pak for interesting conversations and helpful remarks.
This work was started when the first and third authors were participating in the MSRI program \emph{Geometric and Topological Combinatorics}; we thank MSRI for their hospitality.  The third author
was partially supported by the UCLA Dissertation Year Fellowship. The first author would also like to thank San Francisco State University and the second author would like to thank the City University of New York for hosting them as visiting researchers.

\bibliographystyle{plain}  
\bibliography{multiparametric_hardness}

\end{document}